\documentclass{amsart}

\makeindex

\usepackage[all, cmtip]{xy}
\usepackage{amssymb}

\usepackage{algorithmic}
\usepackage{amsmath}
\usepackage{amsthm}
\usepackage{amscd}
\usepackage{amsfonts}
\usepackage{amssymb}
\usepackage[pdftex]{graphicx}
\usepackage{multirow}
\usepackage{color}
\usepackage{epic}
\usepackage{graphicx}
\usepackage{pgf,tikz}

\usetikzlibrary{arrows}

\newtheorem{theorem}{Theorem}[section]
\newtheorem{lemma}[theorem]{Lemma}
\newtheorem{proposition}[theorem]{Proposition}

\newtheorem{question}[theorem]{Question}

\theoremstyle{corollary}
\newtheorem{corollary}[theorem]{Corollary}

\theoremstyle{definition} 
\newtheorem{definition}[theorem]{Definition}

\newtheorem{example}[theorem]{Example}

\newtheorem{claim}[theorem]{Claim}
\theoremstyle{remark}
\newtheorem{remark}[theorem]{Remark}

\numberwithin{equation}{section}

\newcommand{\R}{\mathbb{R}}
\newcommand{\Z}{\mathbb{Z}}

\newcommand{\Q}{\mathbb{Q}}
\def\P{\mathbb{P}}

\def\Pic{\operatorname{Pic}}
\def\Proj{\operatorname{Proj}}
\def\Spec{\operatorname{Spec}}

\def\mult{\operatorname{mult}}
\def\Null{\operatorname{Null}}
\def\Cone{\operatorname{Cone}}
\def\Nef{\operatorname{Nef}}
\def\Ample{\operatorname{Ample}}

\def\cha{\operatorname{char}}
\def\proj{\operatorname{proj}}

\input xy
\xyoption{all}

\title[Characterization of  log del Pezzo pairs]{Characterization of log del Pezzo pairs via  anticanonical models}

\begin{document}

\author{DongSeon Hwang}
\address{Department of Mathematics, Ajou University, Suwon, Korea}
\email{dshwang@ajou.ac.kr}

\author{Jinhyung Park}
\address{Department of Mathematical Sciences, KAIST, Daejeon, Korea}
\email{parkjh13@kaist.ac.kr}

\date{April 23, 2013}

\begin{abstract}
There are several variations of the definition of log del Pezzo pairs in the literature. We define their suitable smooth models, and we show that they are the same. In particular, we obtain a characterization of smooth log del Pezzo pairs in terms of anticanonical models. As applications, we classify non-rational weak log canonical del Pezzo pairs, and we prove that every surface of globally F-regular type is of Fano type.
\end{abstract}

\maketitle
\tableofcontents

\section{Introduction}
Throughout this paper, we work over an algebraically closed field $k$ of arbitrary characteristic. \emph{Del Pezzo surfaces} have been extensively studied from various viewpoints. In particular, considerable effort has been devoted to the classification of log del Pezzo surfaces. Recently, \emph{log del Pezzo pairs} frequently appear motivated by recent developments in the minimal model program. Depending on the problems of interest, there are several variations of the definition of log del Pezzo pairs (see Section \ref{prelim} for our terminology of the precise definitions of various log del Pezzo pairs). However, they are not very well understood from the viewpoint of classification.

The first main theorem of this paper gives a characterization of log del Pezzo pairs in terms of their anticanonical models.

\begin{theorem}\label{main}
Let $X$ be a smooth projective surface with big anticanonical divisor.
Then, $(X, \Delta)$ is a klt (resp. weak lc) del Pezzo pair for some effective $\Q$-divisor  $\Delta$ on $X$  if and only if the anticanonical model of $X$ is a klt (resp. lc) del Pezzo surface.
\end{theorem}

The birational structure of smooth projective surfaces whose anticanonical models are log del Pezzo surfaces is well understood by the comprehensive study of Sakai and B\u{a}descu, and  by the  authors' explicit study on redundant points (see Section \ref{prelim}).  Thus, Theorem \ref{main} reduces the problem of classification of klt (resp. weak lc) del Pezzo pairs to that of klt (resp. lc) del Pezzo surfaces.

Theorem \ref{main} is a part of Theorem \ref{klt} and Theorem \ref{lc}, in which we show that suitable smooth models of various versions of the log del Pezzo pair are the same.

To state Theorem \ref{klt}, we introduce the following five classes.
\begin{enumerate}
\item $LD:=$\{big anticanonical surfaces whose anticanonical models are klt del Pezzo surfaces\}
\item $ES:=$\{$X \mid (X, \Delta_X)$ is a klt del Pezzo pair for some effective $\Q$-divisor $\Delta_X$,  where $X$ is smooth\}
\item $NS:=$\{$X \mid (X, \Delta_X)$ is a klt del Pezzo pair for some simple normal crossing (\emph{snc} for short)    effective $\Q$-divisor $\Delta_X$, where $X$ is smooth\}
\item $EP:=$\{$X \mid f \colon X \rightarrow Y$ is a log resolution of  a klt del Pezzo pair $(Y,\Delta_Y)$ such that $-(K_X + f_*^{-1}\Delta_Y) + f^*(K_Y + \Delta_Y)$  is effective\}
\item $NP:=$\{$X \mid f \colon X \rightarrow Y$ is the minimal resolution, where $(Y,\Delta_Y)$ is a klt del Pezzo pair for some effective snc $\Q$-divisor $\Delta_Y$\}
\end{enumerate}

\begin{theorem}\label{klt}
$LD=ES=NS=EP=NP$.
\end{theorem}

The proof is constructive and the key step is to show that $LD \subset NS$ by explicitly taking an effective $\Q$-divisor $\Delta$ supported on $\Null(P)$, where $-K_X=P+N$ is the Zariski decomposition (see Subsection \ref{LDNS}).

As the class of smooth models for klt del Pezzo surfaces, it is tempting to consider, instead of $LD$, the class
$$  \{ X \mid X \rightarrow Y \textrm{ is the minimal resolution of some klt del Pezzo surface } Y \}.$$
However, there are examples of a big anticanonical rational surface that is not a minimal resolution of any klt del Pezzo surface, even though the anticanonical model is a klt del Pezzo surface (see  \cite[Section 4]{HP}). In general, the anticanonical map in the class $LD$ is decomposed by the minimal resolution of the klt del Pezzo surface followed by a sequence of special blow-ups called redundant blow-ups (\cite[Proposition 4.2]{Sak84} and \cite[Sections 2 and 3]{HP}). Thus, our result implies that in characteristic zero, the singularity on a klt del Pezzo pair is precisely  a quotient singularity modulo a composition of  redundant blow-ups.

The effectivity of the divisor  $-(K_X + f_*^{-1}\Delta_Y)   + f^{*}(K_Y + \Delta_Y)$ in the definition of the class $EP$ is essential in order to make $X$ a big anticanonical surface. Note that there is a klt del Pezzo pair whose log resolution is never a big anticanonical surface (see Example \ref{ex}). On the other hand, by choosing another boundary divisor, we can avoid this situation.

\begin{proposition}\label{newbd}
Let $(Y, \Delta_Y)$ be a klt del Pezzo pair. Then, there exists an effective $\Q$-divisor $\Delta_{Y}'$ and a log resolution $f \colon X \rightarrow Y$ of the pair $(Y, \Delta_{Y}')$ such that $(Y, \Delta_{Y}')$  is a klt del Pezzo pair  and   $-(K_X + f_*^{-1}\Delta_{Y}')   + f^{*}(K_Y + \Delta_Y)$
is effective.
\end{proposition}

In fact, we show that the minimal resolution $X$ of $Y$ is contained in the class $ES$; thus, there is an snc effective $\Q$-divisor $\Delta_X$ such that $(X, \Delta_X)$ is a klt del Pezzo pair by Theorem \ref{klt}. Then, the proposition easily follows by letting $\Delta_{Y}':=f_{*}\Delta_X$. See Section \ref{bd_s} for a complete proof.

\begin{remark}
In general, log resolutions of $(Y, \Delta_{Y})$ are quite different from those of $(Y, \Delta_{Y}')$
\end{remark}

Theorem \ref{klt} cannot be extended directly  to the \emph{lc del Pezzo} case. Indeed, if the anticanonical model of a big anticanonical surface $X$ has a non-klt lc singularity, then there is no effective $\Q$-divisor $\Delta$ such that $(X, \Delta)$ is an lc del Pezzo pair (see Remark \ref{no_lc}). However, the analogous statement holds for the \emph{weak lc del Pezzo} case. For this purpose, we introduce the following five additional classes.

\begin{enumerate}
\item $LCD:=$\{big anticanonical surfaces whose anticanonical models are lc del Pezzo surfaces\}
\item $WES:=$\{$X \mid (X, \Delta_X)$ is a  weak lc del Pezzo pair for some effective $\Q$-divisor $\Delta$, where $X$ is smooth\}
\item $WNS:=$\{$X \mid (X, \Delta_X)$ is a   weak lc del Pezzo pair for some snc effective $\Q$-divisor $\Delta_X$,  where $X$ is smooth\}
\item $WEP:=$\{$X \mid f \colon X \rightarrow Y$ is a log resolution of  a weak lc del Pezzo pair $(Y,\Delta_Y)$ such that
$-(K_X + f_*^{-1}\Delta_Y)  + f^*(K_Y + \Delta_Y)$  is effective\}
\item $WNP:=$\{$X \mid f \colon X \rightarrow Y$ is the minimal resolution, where $(Y,\Delta_Y)$ is a weak lc del Pezzo pair for some snc effective $\Q$-divisor $\Delta_Y$\}
\end{enumerate}

\begin{theorem}\label{lc}
$LCD=WES=WNS=WEP=WNP$.
\end{theorem}

Note that the analogous statement of Proposition \ref{newbd} also holds for the weak lc case (see Remark \ref{lc_prop}).

There exists a smooth projective surface whose anticanonical model is a non-klt lc del Pezzo surface (see \cite[Section 3]{Zha96} for the rational case). In contrast to the \emph{klt} case, there are non-rational lc del Pezzo surfaces, but it turns out that they are very special.

\begin{theorem}\label{non_rat_lc}
Let $Y$ be an lc del Pezzo surface, and let $f \colon X \rightarrow Y$ be the minimal resolution. If $X$ is not rational, then $X$ is a relatively minimal ruled surface with the ruling $\pi \colon X \rightarrow B$ to a smooth elliptic curve, and $f$ contracts a section $C$ that is a smooth elliptic curve. In particular, $Y$ is of Picard number one, and it has exactly one simple elliptic singularity.
\end{theorem}

We closely follow \cite[Proof of Theorem 2]{B83} for the proof of Theorem \ref{non_rat_lc}. Our contribution is to show that the minimal resolution of a non-rational lc del Pezzo surface is a relatively minimal ruled surface.

Del Pezzo surfaces with a non-rational singularity were studied in (\cite{F} and \cite{Sc}), and they were completely classified  in characteristic zero when the Picard number is one (\cite{Ch}). However, there seem to be too many cases for complete classification of rational lc del Pezzo surfaces. See \cite{KT} and \cite{Ko} for recent partial results.

Furthermore, we can classify non-rational weak lc del Pezzo pairs as a consequence of Theorem \ref{non_rat_lc}.

\begin{corollary}\label{non_rat_weak_lc}
Let $(Y, \Delta)$ be a weak lc del Pezzo pair. If $Y$ is not rational, then either of the following holds:
 \begin{enumerate}
 \item $Y$ is a weak lc del Pezzo surface containing exactly one simple elliptic singularity and finitely many rational double points of type $A_n$, or
 \item $Y$ contains at worst finitely many rational double points of type $A_n$.
 \end{enumerate}
Moreover, every minimal resolution of $Y$ can be obtained by a sequence of redundant blow-ups from the minimal resolution of a non-rational lc del Pezzo surface
\end{corollary}

Using the classification,  we characterize weak lc del Pezzo pairs whose Cox rings are finitely generated. In general, the Cox ring of every weak klt Fano pair is finitely generated in   characteristic zero (see \cite[Corollary 1.3.2]{BCHM}, and see also \cite[Corollary 2.15]{HK} for the surface case). However, not all Cox rings of weak lc del Pezzo pairs are finitely generated (e.g., elliptic ruled surfaces). By applying Corollary \ref{non_rat_weak_lc}, we obtain the following.

\begin{corollary}\label{cox}
Let $(Y, \Delta)$ be a weak lc del Pezzo pair. Then, the Cox ring of $Y$ is finitely generated if and only if either $Y$ is rational or $Y$ contains exactly one simple elliptic singularity.
\end{corollary}

Finally, as an application of Theorem \ref{klt}, we consider the following question of Schwede and Smith.

\begin{question}[Question 7.1 of \cite{SS}]
Let $Y$ be a projective variety of globally F-regular type. Does there exist an effective $\Q$-divisor $\Delta$ such that $(Y, \Delta)$ is a klt Fano pair?
\end{question}

For the definition and basic properties of globally F-regular varieties, we refer to \cite{Sm} and \cite{SS}. The question is affirmatively answered for $\mathbb{Q}$-factorial Mori dream spaces (\cite[Theorem 1.2]{GOST}). Here, using Theorem \ref{klt}, we answer the question for surfaces without any assumption.

\begin{theorem}\label{Freg}
Let $Y$ be a projective surface of globally F-regular type. Then, there exists an effective $\Q$-divisor $\Delta$ such that $(Y, \Delta)$ is a klt del Pezzo pair.
\end{theorem}

Schwede and Smith proved that every globally F-regular variety defined over a F-finite field of positive characteristic is of Fano type (\cite[Theorem 1.1]{SS}). However, the choice of the boundary divisor heavily relies on the characteristic of the base field.
On the other hand, our proof can be applied to the positive characteristic case. Here, we emphasize that our explicit choice of the boundary divisor is given by the Zariski decomposition (see Subsection \ref{LDNS} for our choice) so that our choice is independent of the characteristic of the base field.

The converse of Theorem \ref{Freg}  was already established in any dimension by \cite[Theorem 1.2]{SS}. Thus, in principle, Theorem \ref{klt} also gives a classification of projective surfaces of globally F-regular type. However, in positive characteristic, there is a del Pezzo surface which is not globally F-regular (see \cite[p.880]{SS}).


After completing this manuscript, the authors learned that Theorem \ref{Freg} was also obtained independently by Okawa (\cite{Okawa}) and Gongyo and Takagi (\cite{GT}) with different methods.

\section{Preliminaries}\label{prelim}

\subsection{Definitions of log del Pezzo pairs}
We call $(X, \Delta)$ a \emph{log surface pair} if $X$ is a normal projective surface and $\Delta$ is an effective $\Q$-divisor such that $K_X + \Delta$ is $\Q$-Cartier. Let $f \colon Z \rightarrow X$ be a log resolution of $(X, \Delta)$. We have
\begin{equation}\label{disc}
K_Z + f_{*}^{-1}\Delta \equiv f^{*}(K_X + \Delta) + \sum a_i E_i,
\end{equation}
where $E_i$ are exceptional divisors of $f$. Then, a pair $(X, \Delta)$ is called \emph{Kawamata log terminal} (\emph{klt} for short) if every discrepancy $a_i > -1$ and $\lfloor \Delta \rfloor =0$. If $(X, 0)$ is a klt pair, then we say that $X$ has klt singularities. Note that in the case of $\cha(k)=0$, a klt singularity is nothing but a quotient singularity by \cite[Theorem 9.6]{K}. We refer the reader to \cite{KM98} for a complete description of log pairs and their singularities.

\begin{definition}\label{def}\
(1) A klt pair $(X, \Delta)$ is called a \emph{klt del Pezzo pair} if $-(K_X + \Delta)$ is ample, and it is called   a \emph{weak klt del Pezzo pair} if $-(K_X + \Delta)$ is nef and big.\\
(2) A surface $X$ is called a \emph{klt del Pezzo surface} if the pair $(X,0)$ is a klt del Pezzo pair, and it is called a  \emph{weak klt del Pezzo surface} if the pair $(X,0)$ is a weak klt del Pezzo pair.
\end{definition}

A pair $(X, \Delta)$ is called \emph{log canonical} (\emph{lc} for short) if every discrepancy $a_i \geq -1$ of (\ref{disc}) and $b_i \leq 1$, where $\Delta = \sum b_i \Delta_i$ and each $\Delta_i$ is an irreducible component. We can similarly define a \emph{(weak) lc del Pezzo pair} and a \emph{(weak) lc del Pezzo surface} by replacing `klt' with `lc' in Definition \ref{def}.

We remark that every klt singularity is rational (\cite[Theorem 4.12]{KM98}), but not every lc singularity is rational (e.g., simple elliptic singularities).

\subsection{Geometry of big anticanonical surfaces}
A smooth projective surface $X$ is called a \emph{big anticanonical surface} if the anticanonical divisor $-K_X$ is big. Although not every anticanonical ring $R(-K_X):=\bigoplus_{m \geq 0} H^0 (\mathcal{O}_X(-mK_X)$ of a big anticanonical surface $X$ is finitely generated (see \cite[Lemma 14.39]{B01}), we can always define the anticanonical model of $X$ using the Zariski decomposition $-K_X = P+N$. The morphism $f \colon X \rightarrow Y$ given by $|mP|$ for some integer $m \gg 0$ is called the \emph{anticanonical morphism}, and $Y$ is called the \emph{anticanonical model} of $X$. Then, $Y$ is a normal projective surface (see \cite[14.32]{B01}). Note that if $Y$ is an lc del Pezzo surface, then $R(-K_X) \simeq R(-K_Y)$ is finitely generated and $Y \simeq \Proj R(-K_X)$, since $-K_Y$ is ample $\Q$-Cartier.

A point $x$ in a big anticanonical surface $X$ is called \emph{redundant} if $\mult_{x}(N) \geq 1$. The blow-up $\pi \colon \widetilde{X} \rightarrow X$ at a redundant point $x$ is called the \emph{redundant blow-up}, and the exceptional divisor is called the \emph{redundant curve} (see \cite[Definition 4.1]{Sak84} and \cite[Section 2]{HP}).

\begin{proposition}[Proposition 4.2 of \cite{Sak84} and Proposition 4 of \cite{B83}]\label{red}
Let $X$ be a big anticanonical surface. Then, there is a sequence of redundant blow-ups $g \colon X \rightarrow X_0$, where $X_0$ is the minimal resolution of the anticanonical model of $X$.
\end{proposition}

\begin{theorem}[Theorem 4.3 of \cite{Sak84}]\label{ant}
Let $X$ be a big anticanonical rational surface. Then, the anticanonical model of $X$ is a del Pezzo surface with rational singularities. Conversely, the minimal resolution of a del Pezzo surface with rational singularities is a big anticanonical rational surface.
\end{theorem}

\begin{remark}
The analogous statement of Theorem \ref{ant} holds for a non-rational ruled surface under the assumption that $\cha(k)=0$ (see \cite[Theorem 2]{B83} and \cite[Remark in p.237]{B01}). Since we are working in arbitrary characteristic, we should be careful in dealing with the anticanonical model of a non-rational surface.
\end{remark}

For any divisor $D$ on an algebraic surface $X$, we define
$$\Null(D):= \{ x \in X \mid \text{$x$ is a point on an integral curve $C$ such that $C.D=0$}\}.$$
The following lemma will play a crucial role.

\begin{lemma}\label{snc}
Let $X$ be a big anticanonical surface with the Zariski decomposition $-K_X = P+N$, and let $f \colon X \rightarrow Y$ be the anticanonical morphism. If $Y$ has at worst rational singularities or lc singularities, then $\Null(P)$ has an snc support. In particular, the anticanonical morphism $f$ can be regarded as a log resolution of $(Y, 0)$.
\end{lemma}

\begin{proof}
By Proposition \ref{red}, we may assume that $f \colon X \rightarrow Y$ is the minimal resolution. If $Y$ has at worst rational singularities, then the assertion is trivial. Thus, assume that $Y$ has at worst lc singularities. By the classification of surface lc singularities (see \cite[Theorem 4.7]{KM98}), we only have to show that a nodal cubic curve $C$ cannot be contracted by $f$.

Suppose that there is a nodal cubic curve $C$ such that $C.P=0$. Then, $f(C)$ is not a rational singularity, and hence, $X$ is a non-rational ruled surface by Theorem \ref{ant}. Let $\pi \colon X \rightarrow B$ be the ruling to a smooth curve $B$ with genus at least $1$. By \cite[Lemma 14.35]{B01}, every curve $D$ in a fiber of $\pi$ satisfies $p_a(D)=0$, and hence, the nodal cubic curve $C$ cannot be contained in any fibre of $\pi$. Thus, $\pi|_C \colon C \rightarrow B$ is a dominant morphism, and hence, we get a contradiction.
\end{proof}

By the construction of the anticanonical morphism, $\Null(P)$ consists of $f$-exceptional divisors. We have
$$-K_X = f^{*}(-K_Y) - \sum_{E_i \in \Null(P)} a_i E_i.$$
By \cite[Proposition 14.26]{B01}, we obtain the Zariski decomposition $-K_X=P+N$, which is given by $P=f^{*}(-K_Y)$ and $N= - \sum a_i E_i$. If $(Y, 0)$ is a klt (resp. lc) pair, then $0 \leq -a_i < 1$ (resp. $0 \leq -a_i \leq 1$) for every $i$.

\section{KLT del Pezzo pairs}\label{klt_s}
In this section, we prove Theorem \ref{klt} by showing   the following inclusions.
\begin{itemize}
\item $NS \subset ES \subset  LD \subset  NS$
\item $NS \subset  NP \subset  EP \subset  LD$
\end{itemize}

\subsection{$NS \subset ES$}
It is trivial.

\subsection{$ES \subset LD$}\label{ESLD}
This proof is inspired by \cite[Section 3]{TVV11}. Let $(X, \Delta_X)$ be a klt del Pezzo pair, where $X$ is smooth. Then, $X$ is a big anticanonical surface. Let $f \colon X \rightarrow Y$ be the anticanonical morphism. We can define the pull-back of a $\Q$-Weil divisor by using Mumford's intersection theory (see \cite{Mumford}). Note that the $\mathbb{Q}$-divisor
$$-A:= - (K_X + \Delta_X) - f^{*}(-(K_Y + f_{*}\Delta_X)),$$
supported on the exceptional locus of $f$, is $f$-nef. By the negativity lemma (see \cite[Lemma 7.1]{Zar62} or  \cite[Lemma 3.39]{KM98}), $A$ is effective. We have
$$-K_X = f^{*}(-K_{Y}) - \sum a_i E_i,$$
where each $E_i$ is an irreducible component of the whole exceptional divisor. Let $-K_X = P+N$ be the Zariski decomposition. Then, $P=f^{*}(-K_{Y})$ and $N=-\sum a_iE_i$, and hence, we obtain
$$ -A = N - \Delta_X + f^{*} f_{*}\Delta_X.$$

Suppose, to derive a contradiction, that $Y$ is not a klt del Pezzo surface. Then, we have one of the following two cases:
\begin{enumerate}
    \item $-K_Y$ is not $\Q$-Cartier, or
    \item there is a component of $N$ with coefficient at least $1$.
\end{enumerate}
If we assume Case (2), then the divisor
$$\Delta_X =  A  + N   + f^{*} f_{*}\Delta_X$$
contains a component with coefficient at least $1$, since the divisors $A$ and $f^{*} f_{*}\Delta_X$ are effective. Thus, the pair $(X, \Delta_X)$ is not klt, a contradiction. Hence, $-K_Y$ is not $\Q$-Cartier and every component of $N$ has a coefficient greater than $-1$. Then, the pair $(Y, 0)$ is numerically lc, and hence, by \cite[p.112]{KM98}, $-K_Y$ is  $\Q$-Cartier, a contradiction.

\begin{remark}\label{no_lc}
Let $X$ be a big anticanonical surface and let $Y$ be its anticanonical model.
If $Y$ contains a non-klt singularity, then $(X, \Delta)$ is not a lc del Pezzo pair for any effective $\mathbb{Q}$-divisor $\Delta$. Indeed, if $(X, \Delta)$ is an lc del Pezzo pair for some effective $\mathbb{Q}$-divisor $\Delta$, then  $(X, (1-\epsilon ) \Delta)$ is a klt del Pezzo pair for some sufficiently small number $\epsilon >0$; thus, $Y$ is a klt del Pezzo surface by Subsection \ref{ESLD}.
\end{remark}

\subsection{$LD \subset NS$}\label{LDNS}

Let $X$ be a smooth projective surface whose anticanonical model $Y$ is a klt del Pezzo surface. Then, $Y$ contains at worst rational singularities, and hence, $X$ is a big anticanonical rational surface by Theorem \ref{ant}. Let $-K_X = P+N$ be the Zariski decompositon, and let $f \colon X \rightarrow Y$ be the anticanonical morphism. Recall that $P$ is nef and big, and $\lfloor N \rfloor = 0$. By Lemma \ref{snc}, $\Null(P)$ has an snc support. Since $N$ is supported on $\Null(P)$, it has an snc support.  If $P$ is ample, the proof is completed by letting $\Delta=N$. Thus, we now assume that $P$ is nef and big, but not ample.

\begin{claim}\label{cl1}
There is an effective $\Q$-divisor $L$ supported on $\Null(P)$ such that $L.E_i <0$ for every $E_i \in \Null(P)$.
\end{claim}

The assertion $LD \subset NS$ follows from Claim \ref{cl1}. Indeed, let $$\Delta := N + \epsilon L$$ for sufficiently small $\epsilon >0$. Then, $\Delta$ is snc, and the pair $(X, \Delta)$ is klt. It is sufficient to show that $-(K_X + \Delta)=P - \epsilon L$ is ample. By \cite[Proposition 3.3]{Nak07} or \cite[Proposition 2]{TVV11}, the Kleiman-Mori cone $\overline{NE}(X)$ is rational polyhedral. Let $C_1, \ldots, C_n$ be irreducible curves generating $\overline{NE}(X)$, and suppose that $C_1.P=\cdots=C_k.P=0$ and  $C_{k+1}.P, \ldots, C_n.P > 0.$ Then, for every $i$ with $1 \leq i \leq k$, by Claim \ref{cl1}, $L.C_i<0$, so $(P-\epsilon L).C_i >0$. For all $j$ with $k+1 \leq j \leq n$, we may assume that $P.C_j > \epsilon L.C_j$ possibly by choosing smaller $\epsilon$, and hence, we obtain $(P-\epsilon L).C_j >0$. Thus, $P - \epsilon L$ is ample.

\begin{proof}[Proof of Claim \ref{cl1}]
Since the intersection matrix of $\Null(P)$ is negative definite, there exists a $\Q$-divisor $L := \sum_{E_i \in \Null(P)} a_i E_i$ such that $L.E_i = -1$ for every $E_i \in \Null(P)$. Then, $L$ is effective by the negativity lemma, since every irreducible component of $L$ is contracted by the anticanonical morphism $f$. Thus, the proof is complete. See Remark \ref{remark} for an alternative proof.
\end{proof}

\subsection{$NS \subset NP$} Take $Y = X$ and $\Delta_Y = \Delta$; then, the assertion follows.

\subsection{$NP \subset EP$}
Let $f \colon X \rightarrow Y$ be the minimal resolution. Then, $f \colon X \rightarrow Y$ is a log resolution of a klt del Pezzo pair $(Y, \Delta_Y)$ for some snc effective $\Q$-divisor $\Delta_Y$. Using the adjunction formula, for every curve $C$ contracted by $f$, we obtain
$$(-K_X + f^{*}K_Y).C=-K_X.C=C^2 + 2-2p_a(C) \leq 0.$$
Thus, by the negativity lemma, the divisor $-K_X + f^{*}K_Y$ is effective, and hence, the divisor
$$-(K_X + f^{-1}_*\Delta_Y) + f^{*}(K_Y + \Delta_Y) = (-K_X + f^{*}K_Y) + (f^* \Delta_Y - f^{-1}_*\Delta_Y)$$
is also effective.

\subsection{$EP \subset LD$}
Let $(Y, \Delta_Y)$ be a klt del Pezzo pair, and let $f \colon X \rightarrow Y$ be a log resolution of $(Y, \Delta_Y)$. By the assumption,
$$-\sum a_i E_i = -(K_X + \Delta_X) + f^*(K_Y + \Delta_Y)$$
is effective, where $\Delta_X  := f^{-1}_*\Delta_Y$ and each $E_i$ denotes an irreducible component of the $f$-exceptional divisor. Thus, we have $0 \leq -a_i <1$ for every $i$. Since $f^{*}(-(K_Y + \Delta_Y))$ is nef and big, the anticanonical divisor
$$-K_X = f^{*}(-(K_Y+ \Delta_Y)) + \Delta_X - \sum a_i E_i$$
is big, i.e., $X$ is a big anticanonical surface. Note that $\lfloor \Delta_X - \sum a_i E_i \rfloor =0$.

Since $-(K_Y + \Delta_Y)$ is ample, $D:=d \cdot f^{*}(-(K_Y + \Delta_Y))$ is a nef and big $\Z$-divisor on a smooth surface $X$ for some integer $d>0$. By Kodaira's lemma (see \cite[Proposition 2.61]{KM98}), there exists an effective divisor $E$ and ample $\Q$-divisors $A_k$ such that $D \equiv A_k + \frac{1}{k}E$ for a sufficiently large integer $k>0$. For a sufficiently large and divisible integer $m>0$, the divisor $mA_k$ is a very ample $\Z$-divisor. By Bertini's theorem (\cite[Theorem II.8.18]{Har}), we can choose a prime divisor $A \in |mA_k|$. Note that $X$ is a rational surface by \cite[Proposition 3.6]{Nak07}. Then, we have $D = \frac{1}{m}A + \frac{1}{k}E$, since the linear equivalence and the numerical equivalence coincides on $X$. Thus, the effective $\Q$-divisor $\frac{1}{dm}A + \frac{1}{dk}E + \Delta_X - \sum a_i E_i$ is linearly equivalent to $-K_X$, and $\lfloor \frac{1}{dm}A + \frac{1}{dk}E + \Delta_X - \sum a_i E_i \rfloor =0$ possibly by replacing $k$ and $m$ with sufficiently large values.

To derive a contradiction, suppose that the anticanonical model of $X$ contains a non-klt singularity. Let $-K_X = P+N$ be the Zariski decomposition. Recall that $N$ has an irreducible component whose coefficient is at least $1$. By \cite[Lemma 2.4]{Sak84}, we have $|-nK_X|=|nP|+nN$ for a positive integer $n$ such that $nK_X, nP$, and $nN$ are $\Z$-divisors, and hence, for every effective $\Q$-divisor $F$ linearly equivalent to $-K_X$, we have $\lfloor F \rfloor \neq 0$, which is a contradiction. Thus, we complete the proof of Theorem \ref{klt}.

\begin{remark}\label{remark}
We give an alternative proof of Claim \ref{cl1} motivated by \cite[Lemma 6]{CS08}. Let $H$ be an ample divisor on $X$, and let $\epsilon' > 0$ be a sufficiently small rational number. Since $-K_X$ is big, by \cite[Example 1.11]{ELMNP} (more precisely, as a consequence of \cite[Theorem 0.3]{Nm} in $\cha(k)=0$ and \cite[Theorem 0.2]{Ke} in $\cha(k)>0$), we obtain
$$\mathbf{B}_{+}(-K_X) = \Null(P),$$
where $\mathbf{B}_{+}$ denotes the augmented base locus, i.e.,
$$\mathbf{B}_{+}(-K_X) = \mathbf{B}(-K_X - 2 \epsilon' H)$$
for a small positive rational number $\epsilon'$. Here, $\mathbf{B}$ denotes the stable base locus. Thus, we have
$$ \Null(P) =  \mathbf{B}(-K_X - 2 \epsilon' H).$$
By applying the cone theorem (see \cite[Theorem 1.24]{KM98}, and see also \cite[Theorem 4.4]{Ta} for positive characteristic case), we have
$$\overline{NE}(X) = \overline{NE}(X)_{(K_X + \epsilon' H) \geq 0 } + \sum_{i=1}^k \R_{\geq 0}[C_i].$$
For an irreducible curve $C$ satisfying $(K_X + \epsilon' H).C \geq 0$, we have $-(K_X + 2\epsilon' H).C < 0$, and hence, $C \in \Null(P)$. Let $\Cone(\Null(P))$ be the cone generated by the elements of $\Null(P)$. Then, $\overline{NE}(X)_{(K_X + \epsilon H) \geq 0 } \subset \Cone(\Null(P))$. Since $P^2 > 0$, the divisor class of $P$ is not in $\Cone(\Null(P))$. Furthermore, $\Ample(X) \cap \Cone(\Null(P)) = \emptyset$. Consider the hyperplane $H_i$ in $N^1(X):= \Pic(X) \otimes \R$ orthogonal to each extremal ray $\R_{\geq 0}[C_i]$ for $1 \leq i \leq k$. Since $C_i.P=0$, the hyperplane $H_i$ passes through $P$. The hyperplane $H_i$ seperates $C_i$ and the ample cone $\Ample(X)$ in $N^1(X)$. Thus, we can take an ample divisor $A$ such that the line $l$ through $A$ and $P$ meets $\Cone(\Null(P))$. Choose a $\Q$-divisor $L$ in the intersection of $l$ and the interior of $\Cone(\Null(P))$. Then, $L$ is effective and supported in $\Null(P)$. Moreover, $L.C_i <0$ for $1 \leq i \leq k$, since $A.C_i >0$ and $P.C_i=0$. Hence, we have proved Claim \ref{cl1}.

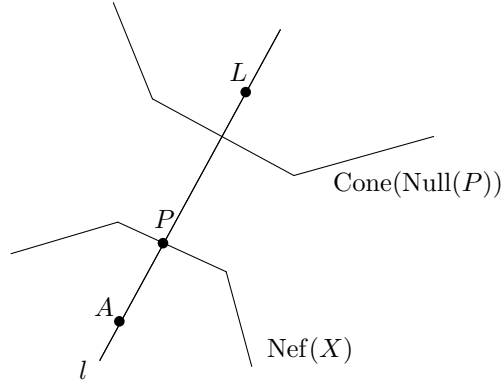
\begin{center}
\begin{tikzpicture}[line cap=round,line join=round,>=triangle 45,x=1.0cm,y=1.0cm]
\clip(-3.1,-1.0) rectangle (3.8,4.5);
\draw (-2.88,1)-- (-1.46,1.42);
\draw (-1.46,1.42)-- (-0.02,0.76);
\draw (-0.02,0.76)-- (0.32,-0.5);
\draw (-1.52,4.34)-- (-1,3.06);
\draw (-1,3.06)-- (0.88,2.04);
\draw (0.88,2.04)-- (2.74,2.56);
\draw (-1.7,-0.42)-- (0.7,3.98);
\draw (0.7,3.98)-- (-1.7,-0.42);
\draw (-0.1,3.66) node[anchor=north west] {$L$};
\draw (-1.1,1.7) node[anchor=north west] {$P$};
\draw (-1.9,0.54) node[anchor=north west] {$A$};
\draw (0.4,0.02) node[anchor=north west] {$\Nef(X)$};
\draw (1.28,2.22) node[anchor=north west] {$\Cone(\Null(P))$};
\draw (-2.1,-0.3) node[anchor=north west] {$l$};
\begin{scriptsize}
\fill [color=black] (-0.86,1.14) circle (2.0pt);
\fill [color=black] (-1.44,0.1) circle (2.0pt);
\fill [color=black] (0.24,3.15) circle (2.0pt);
\end{scriptsize}
\end{tikzpicture}
\end{center}

\end{remark}

As a direct consequence of Theorem \ref{klt}, we give a simple criterion to determine when a smooth projective surface becomes a klt del Pezzo pair using the Zariski decomposition.

\begin{corollary}\label{cri1}
Let $X$ be a big anticanonical surface, and let $-K_X = P+N$ be the Zariski decomposition. Then, there is an effective $\Q$-divisor $\Delta$ such that $(X, \Delta)$ is a klt del Pezzo pair if and only if $\lfloor N \rfloor = 0$.
\end{corollary}

\begin{proof}
Suppose that $(X, \Delta)$ is a klt del Pezzo pair. Let $f \colon X \rightarrow Y$ be the anticanonical morphism. By Theorem \ref{klt}, the anticanonical model $Y$ has at worst klt singularities. We have
$$-K_X = f^{*}(-K_Y) + \sum a_i E_i,$$
where $E_i$ denotes an irreducible component of the $f$-exceptional divisor, and $0 \leq a_i < 1$ for each $i$. Then, we can easily see that the Zariski decomposition $-K_X = P+N$ is given by $P=f^* (-K_Y)$ and $N=\sum a_i E_i$. Thus, $\lfloor N \rfloor = 0$.

Conversely, suppose that  $\lfloor N \rfloor = 0$. Then, $(X, N)$ is a weak klt del Pezzo pair. Thus, we get the conclusion by Kodaira's lemma.
\end{proof}

\section{Weak LC del Pezzo pairs} \label{lc_s}
In this section, we prove Theorem \ref{lc} by showing   the following inclusions.
\begin{itemize}
\item $WNS \subset WES \subset LCD \subset WNS$
\item $WNS \subset WNP \subset  WEP \subset WNS$
\end{itemize}
The proofs of the inclusions $WNS \subset WES, WES \subset LCD, WNS \subset WNP$ and $WNP \subset WEP$ are the same as those of the corresponding cases in the proof of Theorem \ref{klt}, and hence, we only have to show the inclusions $LCD \subset WNS$ and $WEP \subset WNS$.

\subsection{$LCD \subset WNS$}\label{LCNS}
Let $X$ be a big anticanonical surface, and let $f \colon X \rightarrow Y$ be the anticanonical morphism. We assume that $Y$ is a lc del Pezzo surface. Then, we have
$$K_X = f^{*}K_Y + \sum a_i E_i,$$
where each $E_i$ denotes the irreducible component of the $f$-exceptional divisor, and $a_i$ is a rational number such that $0 \leq -a_i \leq 1$ for each $i$. Recall that the Zariski decomposition $-K_X = P+N$ is given by $P=f^*(-K_Y)$ and $N=-\sum a_i E_i$. The assertion $LCD \subset WNS$ immediately follows from Lemma \ref{snc} by letting $\Delta:=N$.

\subsection{$WEP \subset WNS$}
Let $f \colon X \rightarrow Y$ be a log resolution of a weak lc del Pezzo pair $(Y, \Delta_Y)$ such that the $\Q$-divisor
$$D := -(K_X + f_*^{-1}\Delta_Y) + f^{*}(K_Y + \Delta_Y)$$
is effective. Note that every irreducible component of $D$ is $f$-exceptional. Since $f$ is a log resolution, the effective $\Q$-divisor $\Delta := f_*^{-1}\Delta_Y + D$ is snc. Moreover, every coefficient of the irreducible components of $\Delta$ is at most $1$. Note that the divisor
$$-(K_X + \Delta) = -(K_X +  f_*^{-1}\Delta_Y + D) = f^*(-(K_Y+\Delta_Y))$$
is nef and big. Thus, $(X, \Delta)$ is a weak lc del Pezzo pair. Therefore, we complete the proof of Theorem \ref{lc}.

We also have a simple criterion to determine when a smooth projective surface becomes a weak lc del Pezzo pair. The proof is the same as that of Corollary \ref{cri1}.

\begin{corollary}\label{cri2}
Let $X$ be a big anticanonical surface, and let $-K_X = P+N$ be the Zariski decomposition. Then, there is an effective $\Q$-divisor $\Delta$ such that $(X, \Delta)$ is a weak lc del Pezzo pair if and only if every coefficient of an irreducible component of $N$ is at most $1$.
\end{corollary}

\section{Existence of good boundary divisors}\label{bd_s}
First, we give an example of a klt log del Pezzo pair whose log resolutions are never big anticanonical surfaces.

\begin{example}\label{ex}
Consider nine general points $p_1, \ldots, p_9$ in $\P^2$. For $1 \leq i \leq 9$, take three distinct lines $l_i^1, l_i^2, l_i^3$ meeting at $p_i$ but not passing through any of the other eight points.  Let $\Delta := \frac{1}{10} \sum_{i=1}^9 \sum_{j=1}^{3} l_i^j$. Then, $(\P^2, \Delta)$ is a klt del Pezzo pair. Every log resolution $X \rightarrow \P^2$ of  $(\P^2, \Delta)$ factors through the surface obtained by blowing-up at the nine chosen points $p_1, \ldots, p_9$ of $\P^2$, and hence, $X$ is not a big anticanonical surface. In particular, there is no effective $\Q$-divisor $\Delta_X$ on $X$ such that $(X, \Delta_X)$ is a klt (or lc) del Pezzo pair.
\end{example}

We remark that $(\P^2, 0)$ is a klt del Pezzo pair and itself is a log resolution. Proposition \ref{newbd} says that this always happens, i.e., there always exists a \emph{good} boundary divisor for any klt del Pezzo pair such that a log resolution is also a klt del Pezzo pair.

Now, we prove Proposition \ref{newbd}. Let $(Y, \Delta_Y)$ be a klt del Pezzo pair such that $-(K_{Z}+g_{*}^{-1}\Delta_Y)  + g^{*}(K_Y + \Delta_Y)$ is not effective for every log resolution $g \colon Z \rightarrow Y$ of $(Y, \Delta_Y)$. We shall show that there exists an effective $\Q$-divisor $\Delta_Y'$ on $Y$ such that $(Y, \Delta_Y')$ is a klt del Pezzo pair and  $-(K_{X} + f_{*}^{-1}\Delta_{Y}')  + f^{*}(K_Y + \Delta_Y')$ is effective for some log resolution $f \colon X  \rightarrow Y$ of  $(Y, \Delta_Y')$.

Let $f \colon X \rightarrow Y$ be the minimal resolution. Then, we have
$$-K_{X} = -f^{*}K_Y - \sum a_i E_i$$
where each $E_i$ denotes an irreducible component of the whole exceptional divisor of $f$. By the negativity lemma, each $-a_i \geq 0$. Thus, $-K_X$ is  big, since $-K_Y = -(K_Y + \Delta_Y) + \Delta_Y$ is big.  Let $\Delta_X := f_{*}^{-1}\Delta_Y$.  Note that
$$-(K_{X} + \Delta_{X}) = -f^{*}(K_Y + \Delta_Y) + \sum b_i E_i$$
for some rational numbers $b_i$. Then, we have
$$\sum b_i E_i=-K_{X}+f^{*}K_Y  - \Delta_{X}+f^{*}\Delta_Y=-\sum a_i E_i  +( f^{*}\Delta_Y- f_{*}^{-1}\Delta_Y),$$
and hence, every $b_i \geq 0$.

\begin{claim}
$(X, \Delta_X + \sum b_i E_i)$ is a weak klt del Pezzo pair.
\end{claim}

\begin{proof}
Suppose that $b_{i_0} \geq 1$ for some $i_0$. Let $h \colon Z \rightarrow X$ be a log resolution of $(X, \Delta_X + \sum b_i E_i)$. Then, $g:=f \circ h \colon Z \rightarrow Y$ is a log resolution of $(Y, \Delta_Y)$.
\[
\xymatrix{
Z \ar[rr]^{g} \ar[rd]_h && Y\\
&X \ar[ru]_f &}
\]
Let $\Delta_Z := g_{*}^{-1}\Delta_Y = h_{*}^{-1}\Delta_X$.
Then, we have
\begin{displaymath}
\begin{array}{l}
\nonumber  -(K_Z + \Delta_Z) + g^{*}(K_Y + \Delta_Y) \\
\nonumber =  -(K_Z + \Delta_Z) + h^{*}(K_{X} + \Delta_{X}) + h^{*}(-(K_{X} + \Delta_{X}) + f^*(K_Y + \Delta_Y))\\
\nonumber =  -(K_Z + \Delta_Z) + h^{*}(K_{X} + \Delta_{X}) + h^{*}(\sum b_i E_i).
\end{array}
\end{displaymath}
The effective divisor $-(K_Z + \Delta_Z) + h^{*}(K_{X} + \Delta_{X})$ is supported on the $h$-exceptional divisor, and hence, the coefficient of $h_{*}^{-1}E_{i_0}$ in the divisor $-(K_Z + \Delta_Z) + g^{*}(K_Y + \Delta_Y)$ is at least $1$. This is a contradiction to $(Y, \Delta_Y)$ being a klt pair. Thus, $0 \leq b_i < 1$ for every $i$, and hence, $\lfloor \Delta_X + \sum b_i E_i \rfloor =0$.

On the other hand, we have
$$K_Z + (f \circ h)_{*}^{-1}\Delta_Y = (f \circ h)^{*}(K_Y + \Delta_Y) + \sum c_j F_j,$$
where each $F_j$ denotes an irreducible component of the the $f \circ h$-exceptional divisor. Since $(f \circ h)_{*}^{-1}\Delta_Y = h_{*}^{-1} \Delta_X$ and $f^{*}(K_Y + \Delta_Y)=(K_X + \Delta_X) + \sum b_i E_i $, we obtain
$$K_Z + h_{*}^{-1}(\Delta_X + \sum b_i E_i) = h^{*}(K_X + (\Delta_X + \sum b_i E_i)) + \sum c_j F_j + h_{*}^{-1}(\sum b_i E_i).$$
Since every $c_j > -1$, the pair $(X, \Delta_X + \sum b_i E_i)$ is klt. Now, $-(K_{X}+\Delta_X + \sum b_i E_i)= -f^{*}(K_Y + \Delta_Y)$ is nef and big, and hence, we have proved the claim.
\end{proof}

Let $-K_X = P+N$ be the Zariski decomposition. By Kodaira's lemma, there exists an effective $\Q$-divisor $\Delta_{X}'$ on $X$ such that $(X, \Delta_{X}')$ is a klt del Pezzo pair. Thus, there exists an snc effective $\Q$-divisor $\Delta_X''$ supported on $\Null(P)$ such that $(X, \Delta_X'')$ is a klt del Pezzo pair (see Subsections \ref{LDNS} and \ref{LCNS}).

Now, let $\Delta_{Y}':= f_{*}\Delta_X''$. Then, Proposition \ref{newbd} follows by Lemma \ref{adj} and Lemma \ref{image}. Indeed, by Lemma \ref{adj}, the minimal resolution $f \colon X \rightarrow Y$ contracts curves supported on $\Null(P)$. Thus, $f$ is a log resolution of $(Y, f_{*}\Delta_X'')$, since the strict transform $f_{*}^{-1}f_{*}\Delta_X''$ and exceptional divisors are supported on $\Null(P)$ (see Lemma \ref{snc} for sncness of $\Null(P)$). By Lemma \ref{image}, $(Y, f_{*}\Delta_X'')$ is a klt del Pezzo pair. Now, the divisor
$$-(K_X + f_{*}^{-1}f_{*} \Delta_X'') + f^{*}(K_Y + f_{*}\Delta_X'') = -\sum a_i E_i + (f^{*}f_{*}\Delta_X'' - f_{*}^{-1}f_{*}\Delta_X'')$$
is effective. Hence, the proof of Proposition \ref{newbd} is complete.

\begin{lemma}\label{adj}
If $f \colon X \rightarrow Y$ is the minimal resolution, then
every $f$-exceptional curve $E_i$ is contracted by the anticanonical morphism $f' \colon X \rightarrow Y'$.
\end{lemma}

\begin{proof}
Recall that $f'$ is given by $|mP|$ for some integer $m>0$. By the adjunction formula, we have
$$P.E_i + N.E_i = -K_X.E_i =E_i^2 + 2-2p_a(E_i) \leq 0,$$
and hence, we obtain $N.E_i \leq 0$. If $N.E_i <0$, then $E_i$ is contained in the support of $N$, and hence, $E_i.P=0$. If $N.E_i=0$, then $P.E_i=0$. In any case, $E_i$ is contracted by $f'$.
\end{proof}

\begin{lemma}\label{image}
Let $(X, \Delta)$ be a klt del Pezzo pair, and let $f \colon X \rightarrow Y$ be a contraction to a normal projective surface. Then, $(Y, f_{*} \Delta)$ is also a klt del Pezzo pair.
\end{lemma}

\begin{proof}
It is sufficient to show that $(Y, f_{*}\Delta)$ is a klt del Pezzo pair when $f \colon X \rightarrow Y$ contracts only one curve.

First, we claim that $-(K_Y+f_{*} \Delta)$ is ample. Note that the divisor
$$-A:=-(K_{X} + \Delta) + f^{*}(K_{Y} + f_{*} \Delta)$$
is supported on the $f$-exceptional divisor, which is an irreducible curve, and hence, $A^2 < 0$. By the negativity lemma, $A$ is effective. For any irreducible curve $C$ on $Y$, we have
$$-(K_Y + f_{*}\Delta).C = -f^{*}(K_Y + f_{*}\Delta).f^{*}C=-(K_X + \Delta).f^{*}C + A.f^{*}C>0$$
since $A.f^{*}C =0$. Moreover,
$$(-f^{*}(K_Y + f_{*}\Delta)-A)^2 = (-f^{*}(K_Y + f_{*}\Delta))^2 + A^2 = (-(K_X+\Delta))^2 > 0,$$
and thus, $(-f^{*}(K_Y + f_{*}\Delta))^2= (-(K_Y + f_{*}\Delta))^2>0$. The claim follows from the Nakai-Moishezon criterion (\cite[Theorem 1.22]{B01}).

Finally, we show that $(Y, f_{*}\Delta)$ is a klt pair. Let $h \colon Z \rightarrow X$ be a log resolution of $(X, \Delta)$ such that $g:=f \circ h \colon Z \rightarrow Y$ is also a log resolution of $(Y, f_{*}\Delta)$. Note that $h_{*}^{-1}\Delta \geq g_{*}^{-1}f_{*}\Delta$. It suffices to show that every coefficient of an irreducible component of the $\Q$-divisor
$$B:=(K_Z + g_{*}^{-1}f_{*}\Delta) -g^{*}(K_Y + f_{*}\Delta) $$
is greater than $-1$. Since $f_{*}\Delta = g_{*} h_{*}^{-1}\Delta$, the divisor
$$(h_{*}^{-1}\Delta - g_{*}^{-1}f_{*}\Delta) +B =  (K_Z + h_{*}^{-1}\Delta) - g^{*}(K_Y + g_{*}h_{*}^{-1}\Delta) $$
is effective by the negativity lemma. Every coefficient of an irreducible component of $h_{*}^{-1}\Delta - g_{*}^{-1}f_{*}\Delta$ is less than $1$, and hence, every coefficient of an irreducible component $B$ is greater than $-1$.
\end{proof}

\begin{remark}
When $\cha(k)=0$, Lemma \ref{image} is the $\dim=2$ case of \cite[Corollary 3.3]{FG} and \cite[Theorem 2.9]{PS}. Our proof is independent of them, and it works for surfaces in arbitrary characteristic.
\end{remark}

\begin{remark}\label{lc_prop}
Proposition \ref{newbd} still holds even if we replace `klt' with `weak lc', and the proof in this section
also works for this case.
\end{remark}

\section{Classification of non-rational weak lc del Pezzo pairs}\label{classification}
In this section, we prove Theorem \ref{non_rat_lc} and Corollary \ref{non_rat_weak_lc}, which give the classification results, and then, we prove Corollary \ref{cox} to describe when the Cox ring is finitely generated.

\begin{proof}[Proof of Theorem \ref{non_rat_lc}]
Let $Y$ be a non-rational lc del Pezzo surface. The minimal resolution $f \colon X \rightarrow Y$ is the anticanonical morphism, and hence, $X$ is a non-rational ruled surface. By Theorem \ref{ant}, there is at least one irreducible curve $C$ on $X$ such that $f(C)$ is a non-rational singularity. Then, $C$ is not contained in any fiber of the ruling $\pi \colon X \rightarrow B$ to a smooth curve, since every $f$-exceptional irreducible curve in a degenerated fiber of $\pi$ is contracted to a rational singularity by \cite[Lemma 14.35]{B01}. Then, $\pi|_C \colon C \rightarrow B$ is a dominant morphism, and hence, we have $p_a(C) \geq g(B) \geq 1$. Using the classification of surface lc singularities (\cite[Theorem 4.7]{KM98}), we conclude that $C$ is a smooth elliptic curve, and other components of the $f$-exceptional divisor are disjoint from $C$. It follows that $B$ is also a smooth elliptic curve.

We claim that $C$ is the only $f$-exceptional curve not contained in any fiber of $\pi$. Suppose that there is another curve $C'$ not contained in any fiber of $\pi$. By \cite[(14.32.2) in p.232]{B01}, we have $h^1 (\mathcal{O}_{C \cup C'}) \leq h^1(\mathcal{O}_X)=1$. Since $C$ and $C'$ are disjoint, we get a contradiction.

To derive a contradiction, suppose that $X$ is not relatively minimal. Then, there is a smooth rational curve $E$ in a fiber of $\pi$ such that $E^2 \leq -1$ and $C.E \geq 1$. Note that $E$ is not contracted by $f$. Let $-K_X = P+N$ be the Zariski decomposition. Since $P.E>0$ and $N.E \geq C.E \geq 1$, we obtain
$$-K_X.E = P.E + N.E >1;$$
thus, by the adjunction formula, we have $-2=K_X.E + E^2 < -2$, which gives a contradiction .

Finally, we show that $C$ is a section of $\pi$. To this end, we only have to check that $C.F =1$ for any fiber $F$ of $\pi$. Suppose that $C.F \geq 2$ for some fiber $F$ of $\pi$. Since $X$ is relatively minimal, $F^2=0$ and $-K_X.F=2$. Furthermore, $N=C$. Since $P.F>0$, we obtain
$$2=-K_X.F = P.F+ C.F >2,$$
which is a contradiction.
\end{proof}

\begin{proof}[Proof of Corollary \ref{non_rat_weak_lc}]
Let $(Y, \Delta)$ be a weak lc del Pezzo pair. Assume that $Y$ is not rational. By the proof of Proposition \ref{newbd} and Remark \ref{lc_prop}, for the minimal resolution $f \colon X \rightarrow Y$, the anticanonical model $Y_0$ of $X$ is a lc del Pezzo surface. For the minimal resoution $X_0 \rightarrow Y_0$, there is a sequence of redundant blow-ups $g \colon X \rightarrow X_0$ by Proposition \ref{red}. This implies the last assertion.

By contracting a redundant curve $E$, we obtain a morphism $h \colon X \rightarrow X'$. Let $-K_{X'}=P'+N'$ be the Zariski decomposition. In this case, the Zariski decomposition $-K_X = P+N$ is given by $P=h^{*}P'$ and $N=h^{*}N' - E$ (cf. \cite[Corollary 6.7]{Sak84} and \cite[Lemma 3.2]{HP}). By Theorem \ref{non_rat_lc}, $N_0=C$ is a smooth elliptic curve where $-K_{X_0}= P_0+N_0$ is the Zariski decomposition. Thus, we obtain $P=g^{*}P_0$ and $N=g_{*}^{-1}C$. Note that every irreducible $g$-exceptional curve different from $N$ is either a ($-1$)-curve meeting an irreducible curve $N$ or a ($-2$)-curve disjoint from $N$ (cf. \cite[Lemma 3.8]{HP}). Furthermore, any connected component of $g$-exceptional ($-2$)-curves forms a chain.

By Lemma \ref{adj}, every $f$-exceptional curve is contracted by the anticanonical morphism $f' \colon X \rightarrow Y_0$. Thus, we only need to show that if $f$ contracts $N$, then $Y$ is a weak lc del Pezzo surface, i.e., $-K_Y$ is nef. Let $D$ be an irreducible curve on $Y$. Then, we have
$$-K_Y.D = f^{*}(-K_Y).f^{*}D = P.f^{*}D \geq 0,$$
and hence, $-K_Y$ is nef.
\end{proof}

\begin{remark}
If $f$ does not contract a $g$-exceptional curve $G$, then we have $-K_Y.f_{*}G$$=0$, i.e., $-K_Y$ is strictly nef.
\end{remark}

\begin{proof}[Proof of Corollary \ref{cox}]
Let $(Y, \Delta)$ be a weak lc del Pezzo pair. According to Corollary \ref{non_rat_weak_lc}, we have three cases: (1) $Y$ is rational, (2) $Y$ is not rational but has at worst rational double points, and (3) $Y$ contains exactly one simple elliptic singularity. For Case (1), the finite generation of the Cox ring of $Y$ follows from \cite[Theorem 1]{TVV11}. For Case (2), the minimal resolution of $Y$ is obtained by a sequence of blow-ups of an elliptic ruled surface, and hence, the Picard group $\Pic(Y)$ is not finitely generated. Thus, the Cox ring of $Y$ is not finitely generated. For Case (3), the Picard group $\Pic(Y)$ is finitely generated by the following.

\begin{claim}\label{coxcl}
Let $(Y, \Delta)$ be a weak lc del Pezzo pair. If $Y$ has a simple elliptic singularity, then $q(Y)=0$.
\end{claim}

Note that $-K_Y$ is nef and big, and a Cartier divisor. Let $f \colon X \rightarrow Y$ be the minimal resolution, and let $-K_X = P+N$ be the Zariski decomposition. Then, we have $P = f^{*}(-K_Y)$, and hence, $\bigoplus_{m \geq 0} H^0(\mathcal{O}_Y(-mK_Y)) \simeq \bigoplus_{m \geq 0} H^0(\mathcal{O}_X(mP))$. By Corollary \ref{non_rat_weak_lc}, the anticanonical model $Y_0:= \Proj \bigoplus_{m \geq 0} H^0(\mathcal{O}_Y(-mK_Y))$ of $X$ is a lc del Pezzo surface. Thus, $\bigoplus_{m \geq 0} H^0(\mathcal{O}_X(mP))$ is finitely generated. By \cite[Theorem 2.3.15]{L}, which holds over an algebraically closed field of arbitrary characteristic, $-K_Y$ is semiample. Hence, we obtain the projective birational morphism $\pi \colon Y \rightarrow Y_0$. By Theorem \ref{non_rat_lc}, $Y_0$ is $\Q$-factorial. Since $\pi$ is a blow-up of $Y_0$ at some closed subscheme, every $\pi$-exceptional curve is $\Q$-Cartier. Thus, $Y$ is $\Q$-factorial. In this case, the Cox ring of $Y$ is finitely generated if and only if the Kleiman-Mori cone $\overline{NE}(Y)$ is rational polyhedral and every nef divisor is semiample (see \cite[Proposition 2.9]{HK}).

First, we show that the Kleiman-Mori cone $\overline{NE}(Y)$ is rational polyhedral. By Kodaira's lemma and Bertini's theorem, we can find a very ample prime divisor $A$ and an effective $\Q$-divisor $E$ such that  $-K_Y  = \frac{1}{m}A + \frac{1}{k}E$ for some sufficiently large positive integers $m$ and $k$. Let $\Delta' = \frac{1}{k}E$. For a sufficiently small rational number $\epsilon >0$ and for any curve $C$, we have
$$(K_Y + \Delta' + \epsilon A).C = -(\frac{1}{m}-\epsilon)A.C <0.$$
By the cone theorem (\cite[Theorem 3.2]{Fu} and \cite[Theorem 4.4]{Ta}), the assertion follows.

Now, let $D$ be a nef $\Q$-divisor on $Y$. Recall that  $-K_Y  = \frac{1}{m}A + \frac{1}{k}E$.  Let $\Delta'' =  \frac{1}{m}A + \frac{1}{k}E + \epsilon' D$  for a sufficiently small rational number $\epsilon' >0$. Then, we have $K_Y + \Delta'' = \epsilon' D$. By the abundance theorem for surfaces (\cite[Corollary 1.2]{Fu} and \cite[Theorem 0.2]{Ta}), $D$ is semiample. Thus, we complete the proof.
\end{proof}

\begin{proof}[Proof of Claim \ref{coxcl}]
Let $f \colon X \rightarrow Y$ be the minimal resolution. Then, $f_{*} \mathcal{O}_X = \mathcal{O}_Y$ and $R^q f_{*} \mathcal{O}_X = 0$ for $q \geq 2$. Consider the exact sequence induced by the Leray spectral sequence associated with the morphism $f$.
$$0 \rightarrow H^1(\mathcal{O}_Y) \rightarrow H^1(\mathcal{O}_X) \xrightarrow{\epsilon} H^0(R^1 f_{*} \mathcal{O}_X) \rightarrow \cdots$$
By \cite[Proposition 4.2.1]{G1}, we have $H^0(R^1 f_{*} \mathcal{O}_X) = \proj \lim_Z H^1 (\mathcal{O}_Z)$, where $Z$ runs over all effective divisors on $X$, whose supports are contained in the $f$-exceptional divisor. The map $\epsilon$ is induced by the restriction maps $\epsilon_Z \colon H^1(\mathcal{O}_X) \rightarrow H^1(\mathcal{O}_Z)$. Note that $h^0(R^1 f_{*} \mathcal{O}_X)=1$ and the map $H^1 (\mathcal{O}_Z) \rightarrow H^1(\mathcal{O}_{Z'})$ of each inverse limit is surjective. Let $C$ be an elliptic curve in $X$ contracting to the elliptic singularity on $Y$. Then, the restriction map $\epsilon_C \colon H^1(\mathcal{O}_X) \rightarrow H^1(\mathcal{O}_C)$ is an isomorphism, and hence, $\epsilon$ is an isomorphism. Thus, $q(Y)=0$.
\end{proof}

\section{Surfaces of globally F-regular type and klt del Pezzo pairs}

In this section, we prove Theorem \ref{Freg}. First, we briefly recall the definitions of the notion of F-regularity. See \cite{Sm} and \cite{SS} for further details.

Let $k$ be a perfect field of positive characteristic. A finitely generated $k$-algebra $R$ is called \emph{strongly F-regular} if for every nonzero element $c \in R$, there exists an integer $e >0$ such that $cF^e \colon R \rightarrow F^e_{*} R$ splits in the category of $R$-modules, where $F^e \colon R \rightarrow R$ is the $e$-th iterated Frobenius map. A projective variety $Y$ over $k$ is called \emph{globally F-regular} if the section ring $R(Y, H) := \bigoplus_{m \geq 0 } H^0 (Y, \mathcal{O}_Y (mH))$ for some ample divisor $H$ is strongly F-regular.

\begin{remark}\label{Frem}
(1) Every local ring $\mathcal{O}_{Y, y}$ of a globally F-regular variety $Y$ over $k$ is strongly F-regular for every $y \in Y$ (see \cite[(2.2)]{Sm}).\\
(2) Originally, the notion of F-regularity was defined over any F-finite field of positive characteristic (\cite{Sm}).
\end{remark}

Now, let  $k$ be an algebraically closed field of characteristic zero, and let $Y$ be a projective variety over $k$. Here, we briefly explain the method of reduction to characteristic $p$. For a finitely generated $\Z$-subalgebra $A$ of $k$, we can construct a projective scheme $Y_A$ of finite type over $A$ such that $Y =  Y_A \times_A \Spec k$. Then, for every closed point $\mu \in \Spec A$, the fiber $Y_{\mu}$ of $Y_A$ is a projective variety over the residue field $k(\mu)$ which is a finite field.
\[
\xymatrix{
Y =  Y_A \times_A \Spec k \ar[r] \ar[d] & Y_A  \ar[d] & \ar@{_{(}->}[l]  Y_{\mu} \ar[d] \\
\Spec k \ar[r] & \Spec A & \ar@{_{(}->}[l] \Spec k(\mu)
}
\]
A projective variety $Y$ over $k$ is said to be \emph{of globally F-regular type} if there exists a dense open subset $S$ of closed points in $\Spec A$ such that $Y_{\mu}$ is globally F-regular over $k(\mu)$ for every closed point $\mu \in S$. Let $\{ \mathcal{F}_i \}$ be any finite collection of coherent $\mathcal{O}_Y$-modules. By the generic flatness (\cite[Theorem 6.9.1]{G2}), we may assume that $Y_A$ and all members of the collection $\{ \mathcal{F}_{i,A} \}$ are flat over $A$ by possibly enlarging $A$, where each $\mathcal{F}_i$ is the pull-back of $\mathcal{F}_{i,A}$ via the projection $Y \rightarrow Y_A$. Let $g \colon X \rightarrow Z$ be a morphism of projective varieties over $k$. By possibly enlarging $A$, we may assume that $g$ is induced by a morphism $g_A \colon X_A \rightarrow Z_A$ between projective schemes of finite type over $A$ so that we obtain a morphism $g_{\mu} \colon X_{\mu} \rightarrow Z_{\mu}$ between projective varieties of finite type over $k(\mu)$ for every closed point $\mu \in \Spec A$.

Now, we collect some useful facts on varieties of globally F-regular type over an algebraically closed field $k$ of characteristic zero.

\begin{lemma}\label{Fklt}
Every $\mathbb{Q}$-Gorenstein projective  variety of globally F-regular type is normal and Cohen-Macaulay, and it contains at worst klt singularities.
\end{lemma}

\begin{proof}
The assertion follows from \cite[(5.3)]{Sm}.
\end{proof}

\begin{lemma}\label{Fmor}
Let $g \colon X \rightarrow Z$ be a birational morphism between normal projective varieties. Then, the following hold.
 \begin{enumerate}
    \item If $X$ is of globally F-regular type, then so is $Z$.
    \item Assume that $Z$ is $\Q$-Gorenstein and $-K_X + g^{*}K_Z$ is effective. Then, $X$ is of globally F-regular if and only if so is $Z$.
 \end{enumerate}
\end{lemma}

\begin{proof}
Recall that for a finitely generated $\Z$-subalgebra $A$ of $k$, we obtain a birational morphism $g_{\mu} \colon X_{\mu} \rightarrow Z_{\mu}$ between normal projective varieties of finite type over $k(\mu)$ for every closed point $\mu \in \Spec A$. Recall that by the definition, if $X$ is of globally F-regular type, then $X_{\mu}$ is globally F-regular for every closed point of some dense open subset in $\Spec A$. By \cite[Proposition 1.2]{HWY}, $Z_{\mu}$ is globally F-regular for every $\mu$, so $Z$ is of globally F-regular type. Similarly, we can show assertion (2) by using \cite[Proposition 1.4]{HWY} since the assumptions still hold after the reduction to characteristic $p$.
\end{proof}

\begin{lemma}\label{Qfact}
Every projective surface of globally F-regular type is $\mathbb{Q}$-factorial.
\end{lemma}

\begin{proof}
Every projective surface of globally F-regular type contains at worst rational singularities by \cite[Corollary 5.3]{Sm}, so it is $\Q$-factorial by \cite[Theorem 4.6]{B01}.
\end{proof}

\begin{lemma}\label{Fbig}
Let $Y$ be a smooth projective surface of globally F-regular type. Then, $Y$ is rational and the anticanonical divisor $-K_Y$ is big.
\end{lemma}

\begin{proof}
First, we claim that $-K_Y$ is pseudo-effective. For any ample divisor $L$ on $Y$, there exists a finitely generated $\Z$-subalgebra $A$ of $k$ such that $L_{\mu}$ is an ample divisor on $Y_{\mu}$ for every closed point $\mu \in \Spec A$. Since $-K_{Y_{\mu}}$ is big, we have
$$-K_Y.L = -K_{Y_{\mu}}.L_{\mu}>0,$$
which proves the claim.  Note that we also proved that $K_Y$ is not pseudo-effective.

By \cite[Corollary 5.5]{Sm}, we have $h^1(\mathcal{O}_Y)=0$. Thus, $Y$ is rational by Castelnuovo's rationality criterion. In this case, by \cite[Lemma 3.1]{Sak84}, $-K_Y$ is an effective $\Q$-divisor. So, we can take the Zariski decomposition
$-K_Y = P+N$ where the divisors $P$ and $N$ are $\mathbb{Q}$-divisors.

Suppose that $-K_Y$ is not big. Then, since $\kappa(-K_Y) = 0$ or $1$, we have either $P=0$ or $P \neq 0$.

(1) The case  $P=0$. By \cite[Theorem 3.4]{Sak84}, there is an effective divisor $D \in |-mK_Y|$ for some integer $m>0$ such that the intersection matrix of irreducible components of $D$ is negative definite. There exists a finitely generated $\Z$-subalgebra $A$ of $k$ such that for every irreducible component $C$ of $D$, the effective divisor $C_{\mu}$ is an integral curve on $Y_{\mu}$ for every closed point $\mu \in \Spec A$. Then, $\kappa(D_{\mu})=0$ since the intersection matrix of irreducible components of $D_{\mu}$ is negative definite. On the other hand, $-K_{Y_{\mu}}$ is big, which is a contradiction.

(2) The case $P \neq 0$. By \cite[Theorem 3.4]{Sak84}, there is a birational morphism $Y \rightarrow \bar{Y}$ such that $-K_{\bar{Y}}$ is nef and $\bar{Y}$ is obtained from $\P^2$ by blowing up $9$ points. Thus, $(-K_{\bar{Y}})^2=0$. Since $-K_{\bar{Y}}$ is an effective $\Q$-divisor, there exists a finitely generated $\Z$-subalgebra $A$ of $k$ such that $-K_{Y_{1, \mu}}$ is nef for every closed point $\mu \in \Spec A$. Since $\bar{Y}$ is of globally F-regular type by Lemma \ref{Fmor}, $-K_{\bar{Y}_{\mu}}$ is nef and big. However, we have
$$(-K_{\bar{Y}_{\mu}})^2=(-K_{\bar{Y}})^2=0,$$
so we get a contradiction. Therefore, $-K_Y$ is big.
\end{proof}

\begin{remark}
An idea from \cite{GT} was used in proving pseudo-effecitivity of $-K_Y$.
\end{remark}

Theorem \ref{Freg} directly follows from Lemma \ref{Fbig} by using  \cite[Theorem 1]{TVV11} and \cite[Theorem 1.2]{GOST}. Here, we give a direct proof by using Theorem \ref{klt}.

\begin{proof}[Proof of Theorem \ref{Freg}]
By Lemma \ref{Qfact}, $Y$ is $\mathbb{Q}$-factorial, so it is $\mathbb{Q}$-Gorenstein. Let $f \colon X \rightarrow Y$ be the minimal resolution. Then, $-K_X + f^{*}K_Y$ is effective. By Lemma \ref{Fmor} (2), $X$ is of globally F-regular, so by Lemma \ref{Fbig}, the anticanonical divisor $-K_X$ is big. By applying Lemma \ref{Fmor} (1), we conclude that the anticanonical model $Y_0$ of $X$ is also of globally F-regular. Thus, $Y_0$ is a klt del Pezzo surface by Lemma \ref{Fklt}. By Theorem \ref{klt}, the big anticanonical surface $X$ belongs to the class $LD$ which is shown to be the same as the class $NS$, and then, by applying Lemma \ref{image}, we get the conclusion.
\end{proof}

\begin{remark}
The same argument can be applied in positive characteristic, since all the Lemmas in this section also hold for globally F-regular surfaces.
\end{remark}

\smallskip

\noindent\textbf{Acknowledgments.} The authors would like to thank Shinnosuke Okawa and Yoshinori Gongyo for pointing out a gap in the previous version of Lemma \ref{Fbig}, sending the authors their preprints, and useful comments. DongSeon Hwang also thanks Junmyeong Jang and Christian Liedke for useful comments and discussions. Jinhyung Park would like to thank his advisor Sijong Kwak for warm encouragement.

\end{document}